\documentclass[12pt]{amsart}
\usepackage{enumerate}
\usepackage{amssymb}
\usepackage{graphicx}
\usepackage{amscd, color}
\usepackage{amsmath}
\usepackage{amsfonts}
\usepackage[english]{babel}
\usepackage{epsfig}
\usepackage{latexsym,amssymb,amsmath}
\usepackage{mathrsfs}
\usepackage{enumerate}
\usepackage{hyperref}
\usepackage{graphicx,graphics}
\usepackage{amsfonts}
\usepackage{amsthm}

\usepackage{fancyvrb}
\usepackage{alltt}
\usepackage{fancyhdr}
\usepackage[parfill]{parskip}
\usepackage[shortlabels]{enumitem}
\usepackage{afterpage}
\usepackage{setspace}

\input{xy}
\xyoption{all}
\usepackage[utf8]{inputenc}
\usepackage{color}

\usepackage{tikz-cd}

\numberwithin{equation}{section}
\newtheorem{theorem}{Theorem}[section]

\newtheorem{proposition}[theorem]{Proposition}

\newtheorem{example}[theorem]{Example}
\newtheorem{remark}[theorem]{Remark}

\newtheorem{lemma}[theorem]{Lemma}

\newtheorem{definition}[theorem]{Definition}

\usepackage{tikz}
\usepackage{hyperref}
\usetikzlibrary{matrix}

\DeclareMathOperator{\Int}{Int}

\begin{document}

\title{A Jordan Curve Theorem for 2-dimensional Tilings}

\author{Diego Fajardo-Rojas, Natalia Jonard-P\'erez}



\thanks{Conacyt grant 252849 (M\'exico) and by  PAPIIT grant IN115819  (UNAM, M\'exico).}

\address{Departamento de  Matem\'aticas,
Facultad de Ciencias, Universidad Nacional Aut\'onoma de M\'exico, 04510 Ciudad de M\'exico, M\'exico.}

\email{(D.\,Fajardo-Rojas) dfajardorojas@gmail.com}
\email{(N.\,Jonard-P\'erez) nat@ciencias.unam.mx}

\maketitle

\begin{abstract}
    The classical Jordan curve theorem for digital curves asserts that the Jordan curve theorem remains valid in the Khalimsky plane. Since the Khalimsky plane is a quotient space of $\mathbb R^2$ induced by a tiling of squares, it is natural to ask for which other tilings of the plane it is possible to obtain a similar result. In this paper we prove a Jordan curve theorem which is valid for every locally finite tiling of $\mathbb R^2$.  
    As a corollary of our result, we generalize some classical Jordan curve theorems for grids of points, including Rosenfeld's  theorem.

\end{abstract}

\section{Introduction}
The Jordan curve theorem asserts that a simple closed curve divides the plane into two connected components, one of these components is bounded whereas the other one is not. This theorem was proved by Camille Jordan in 1887 in his book \textit{Cours d'analyse} \cite{Jordan}.

During the decade of the seventies, Azriel Rosenfeld published a series of articles \cite{Rosenfeld1, Rosenfeld2, Rosenfeld3, Rosenfeld4, Rosenfeld5, Rosenfeld6} wherein he studies connectedness properties of the grid of points with integer coordinates $\mathbb{Z}^{2}$. For any point $(n,m)\in\mathbb{Z}^{2}$ Rosenfeld defines its \textit{4-neighbors} as the four points $(n,m\pm 1)$ and $(n\pm 1, m)$, whilst the four points $(n\pm 1,m\pm 1)$ as well as the 4-neighbors are its \textit{8-neighbors}. For $k=4$ or $k=8$, a \textit{k-path} is a finite sequence of points $(x_{0}, \dots, x_{n})$ in $\mathbb{Z}^{2}$ such that for every $i\in\{1, \dots, n\}$, $x_{i-1}$ is a $k$-neighbor of $x_{i}$. A subset $S$ of $\mathbb{Z}^{2}$ is \textit{k-connected} if there is a $k$-path between any two elements of $S$. A \textit{k-component} of $S$ is a maximal $k$-connected subset. Lastly, a \textit{simple closed k-path} is a $k$-connected set $J$ which contains exactly two $k$-neighbors for each of its points.  In this case we can describe $J$ as a $k$-path $(x_{1}, \dots, x_{n})$, where  $x_i$ is a $k$-neighbor of $x_{j}$ iff $i\equiv j \pm 1$ (modulo $n$).   In  \cite{Rosenfeld2}, Rosenfeld proves the following discrete Jordan curve theorem.

\begin{theorem}\label{teo:Rosenfeld}
Let $J\subset\mathbb{Z}^{2}$ be a simple closed $4$-path with at least five points. Then $\mathbb{Z}^{2}\setminus J$ has exactly two $8$-components. 
\end{theorem}

\begin{figure}[htp]
    \centering
    \includegraphics[width=12cm]{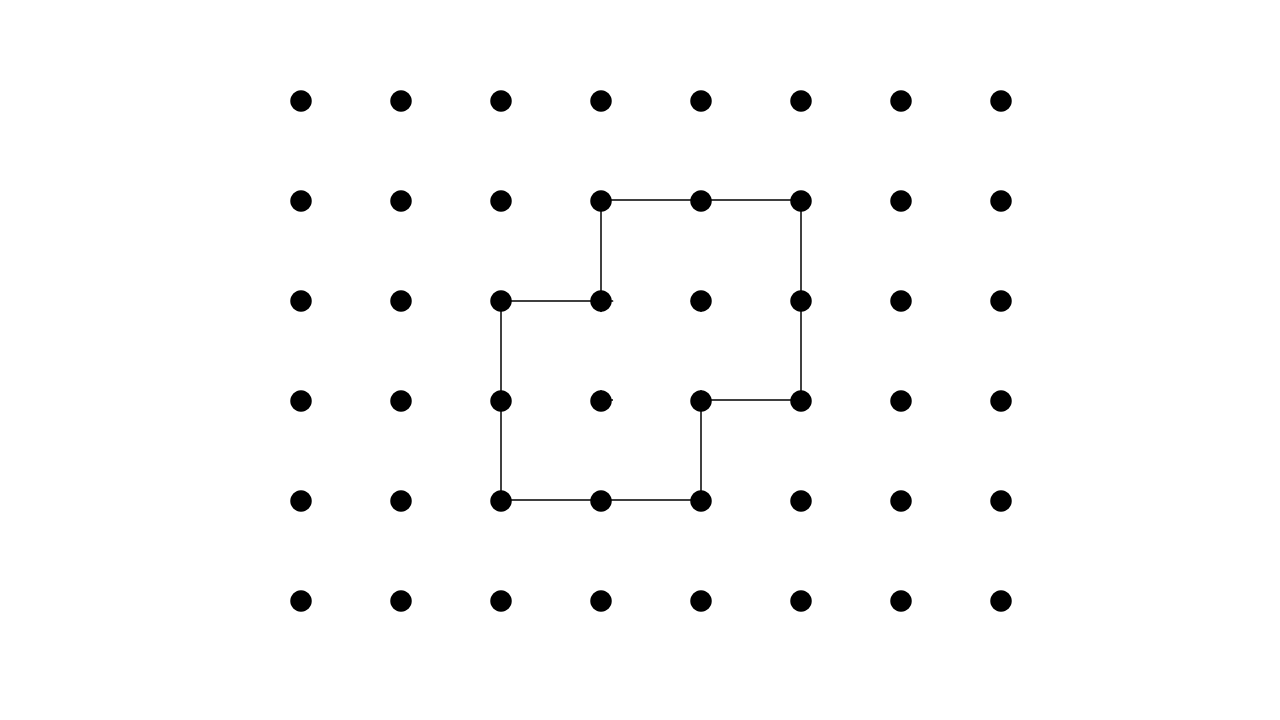}
    \caption{A simple closed $4$-path in $\mathbb{Z}^2$.}
    \label{fig:4-curva}
\end{figure}

In the same decade, Efim Khalimsky developed a different approach  for studying topological properties of the sets $\mathbb{Z}$ and $\mathbb{Z}^{2}$ \cite{Khalimsky1, Khalimsky2}, endowing them with topologies that captured the proximity of its elements. These sets, together with the topologies proposed by Khalimsky, are known as the digital line (or Khalimsky line) and the digital plane (or Khalimsky plane), respectively. Khalimsky also proved a Jordan curve theorem for the digital plane (see \cite{Khalimsky1, Khalimsky2}). We explain this theorem later, in Section~\ref{sec:preliminares} ( Theorem~ \ref{teo:curvadejordanplanodigital}).

There are other versions of Jordan curve theorems for  grids of points that are not the usual squared grid. For example, the case where the points are configured in a hexagonal grid has been studied in \cite[Theorem 26]{Kopperman} and   \cite{Kong2}. In \cite{Neumann-Lara}, V. Neumann-Lara and R. Wilson presented an analogue to the Jordan curve theorem in the context of graph theory (see Theorem~\ref{teo:jordangraficas}). This result will be a key tool for the proof of our main theorem.

One way to define the topology of the digital plane is by an equivalence relation in $\mathbb{R}^2$ resulting of identifying the points in the same edge or the same face of a tiling of the plane by squares of the same size (see equation~\ref{e: square surjection} in Section~\ref{sec:preliminares}). This approach for describing the topology of the digital plane leads us to question if it is possible to generalize Khalimsky's Jordan curve theorem for every quotient space of $\mathbb R^2$ obtained by doing a similar identification in any nice enough tiling. Our main result, Theorem \ref{teo:jordanteselaciones},  states that this generalization is possible if the tiling is locally finite.

Our paper is organized as follows. In Section~\ref{sec:preliminares} we recall all basic notions concerning Alexandrov spaces and tilings. In Section~\ref{sec_ Digital Topology of a tiling}, we 
introduce what the digital version of a tiling is and we prove some basic properties of these topological spaces. In Section~\ref{Sec:main theorem}, we introduce the definition of a digital Jordan curve and we prove our version of the Jordan curve theorem for tilings (Theorem~\ref{teo:jordanteselaciones}). To finish the paper, in Section~\ref{sec:open and closed} we explore some basic properties of closed and open digital Jordan curves, that will be used in Section~\ref{Sec:Well behaved} to obtain conditions in order to guarantee that  a digital Jordan curve encloses a face of the tiling.

\section{Preliminaries}\label{sec:preliminares}

\subsection{Alexandrov Topologies}

Let $(X,\tau)$ be a topological space. The topology $\tau$ is called an \textit{Alexandrov topology} if the intersection of an arbitrary family of open sets is open.  In this case we say that $(X,\tau)$ is an \textit{Alexandrov discrete space}.

It is not difficult to see that $(X,\tau)$ is an Alexandrov discrete space  iff each point $x\in X$ has a smallest neighborhood, which we denote by $N(x)$. Clearly, $N(x)$ is the intersection of all open neighborhoods of $x$ and therefore $N(x)$ is open. 
Also, every subspace of an Alexandrov space is  an Alexandrov space. 

The first important example of an Alexandrov space concerning the contents of this work is the digital line, also known as the Khalimsky line. 

\begin{example}
For every $n\in\mathbb{Z}$, let \[   
N(n) = 
     \begin{cases}
       \{n\} &\quad\text{if $n$ is odd,}\\
       \{n-1, n, n+1\} &\quad\text{if $n$ is even.}\\
       
     \end{cases}
\]
Then the set of integers $\mathbb{Z}$, together with the topology $\tau_{K}$ given by the base $\mathcal{B}=\{N(n)\mid n\in\mathbb{Z}\}$ is known as the \textit{digital line}. 
\end{example}

 Another important example is the digital plane, also known as the Khalimsky plane. 
 
 \begin{example}
The \textit{Khalimsky plane} (or \textit{digital plane}) is the product space of the Khalimsky line $(\mathbb Z,\tau_K)\times(\mathbb Z,\tau_{K})$. This product topology coincides with the quotient topology generated by the surjection $p:\mathbb{R}^{2}\rightarrow\mathbb{Z}^{2}$ defined by:

\begin{equation}\label{e: square surjection}
p(x,y) = 
     \begin{cases}
       (2n+1,2m+1) &\quad\text{if there are $n,m\in\mathbb{N}$ such that} \\ &\quad(x,y)\in(2n,2n+2)\times(2m,2m+2),\\
       \\
        (2n,2m+1) &\quad\text{if there are $n,m\in\mathbb{N}$ such that} \\ &\quad x=2n,\hspace{.2cm}y\in(2m,2m+2),\\
       \\
        (2n+1,2m) &\quad\text{if there are $n,m\in\mathbb{N}$ such that} \\ &\quad x\in(2n,2n+2),\hspace{.2cm}y=2m,\\
       \\
        (2n,2m) &\quad\text{if there are $n,m\in\mathbb{N}$ such that} \\ &\quad x=2n,\hspace{.2cm}y=2m.\\
     \end{cases}
\end{equation}  

 It is worth highlighting that the quotient map $p$ induces an equivalence relation that identifies the interior, the edges (without vertices) and the vertices of the rectangles $[2n,2n+2]\times[2m,2m+2]$ in the plane. Thus, another way to understand the topology of the digital plane is to think of this space as the set of equivalence classes of this equivalence relation together with the quotient topology induced by $\mathbb{R}^{2}$. This description of the digital plane is depicted in Figure \ref{fig:planodigitalreales} . As we already mentioned, the digital plane is an Alexandrov space which can be readily verified by noting that

\[
N(n,m) = 
     \begin{cases}
       \{(n,m)\} &\quad\text{if $2\not\vert n$, $2\not\vert m$}\\
       \{n-1, n, n+1\}\times\{m-1, m, m+1\} &\quad\text{if $2\vert n$, $2\vert m$}, \\
       \{n\}\times\{m-1, m, m+1\} &\quad\text{if $2\not\vert n$, $2\vert m$}, \\
       \{n-1, n, n+1\}\times  \{m\}&\quad\text{if $2\vert n$, $2\not\vert m$}. \\
     \end{cases}
\]

\begin{figure}[htp]
    \centering
    \includegraphics[width=12cm]{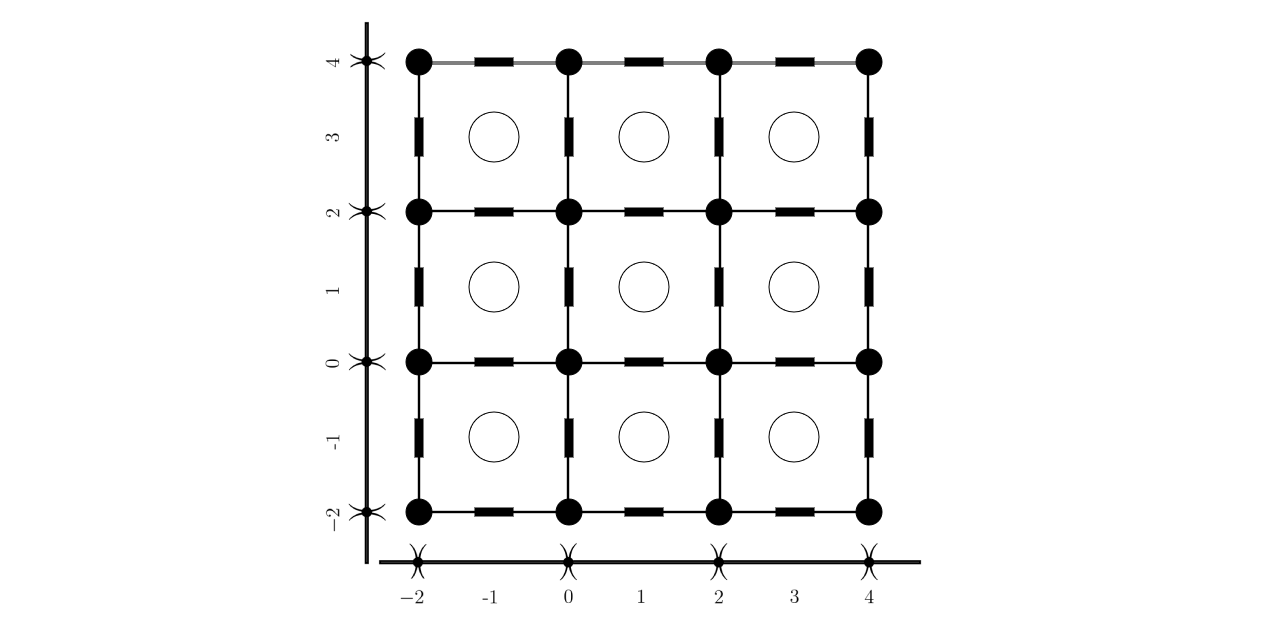}
    \caption{The digital plane as equivalence classes of $\mathbb{R}^{2}$.}
    \label{fig:planodigitalreales}
\end{figure}

\end{example}

An important tool to understand the connectedness of an Alexandrov space is its connectedness graph. 
\begin{definition}\label{graficadeconexidad}
Let $X$ be a topological space. The connectedness graph of $X$ is the graph $G=(V_{G},E_{G})$ where $V_{G}=X$ and $\{x,y\}\in E_{G}$ if and only if $\{x,y\}$ is connected.
\end{definition}
We refer the reader to \cite{Bondy} and \cite{Diestel} for any unknown notion concerning graph theory. 

Let $X$ be a topological space. For every $x\in X$ the  \textit{adjacency set} of $x$ is the set $$\mathscr{A}(x)=\big\{y\in X~\vert~x\neq y,\hspace{.1cm}\{x,y\}\hspace{.1cm}\text{is connected}\big\}.$$

Observe that $x\in\mathscr{A}(y)$ if and only if $y\in\mathscr{A}(x)$. In this case we say that $x$ and $y$ are adjacent. 

Given two points $x,y\in X$,  we define a \textit{digital path} from $x$ to $y$ as a finite sequence of elements of $X$,  $(x_{0}, x_{1}, \dots, x_{n})$, such that $x=x_{0}$, $y=x_{n}$ and for every $i\in\{0, 1, \dots, n-1\}$, $x_{i}$ and $x_{i+1}$ are adjacent. We say that $X$ is \textit{digitally pathwise connected} if for every $x$, $y\in X$ there is a digital path from $x$ to $y$. Lastly, a digital path $(x_{0}, \dots, x_{n})$ is a \textit{digital arc} if $n=1$ or if there is a homeomorphism from a finite interval $I$ of the digital line into $\{x_{0}, \dots, x_{n}\}$. We say that $X$ is \textit{digitally arcwise connected} if for every $x$, $y\in X$ there is a digital arc from $x$ to $y$.

From the definition of the connectedness graph and the definition of a digital path it is clear that $(x_0, \dots, x_n)$ is a digital path in $X$ if and only if $(x_0, \dots, x_n)$ is a path in the connectedness graph $G$ of $X$. The following remark specifies the type of subgraphs in the connectedness graph of $X$ corresponding to digital arcs. For a proof of the non-trivial implication, see e.g. \cite[Theorem 5.6]{Melin}.

\begin{remark}\label{obs:arcocamino}
Let $X$ be a topological space. Then $(x_0, x_{1}, \dots, x_{n})$ is a digital arc if and only if $(x_0, x_{1}, \dots, x_{n})$ is an induced simple path in the connectedness graph of $X$.
\end{remark}

The following theorem summarizes important results about connectedness in  Alexandrov discrete spaces. Its proof can be found in   \cite[Theorem 3.2]{Khalimsky3} and  \cite[Lemma 20]{Kopperman}.

\begin{theorem}\label{teo:resumenconexidad}
Let $X$ be an Alexandrov discrete space.
\begin{enumerate}[\rm(1)]
    \item For every $x\in X$, $\mathscr{A}(x)=\big(N(x)\cup\overline{\{x\}}\big)\setminus\{x\}.$
    \item The following are equivalent:
    \begin{enumerate}[\rm a)]
        \item X is connected.
        \item X is digitally pathwise connected.
        \item X is digitally arcwise connected.
        \item The connectedness graph of $X$ is connected.
    \end{enumerate}
    \item The connected components of $X$ are open and closed.
\end{enumerate}
\end{theorem}

Finally, we enunciate Khalimsky's Jordan curve theorem for the digital plane. First, we say that a subset $J$ of the digital plane is a \textit{Jordan curve in the digital plane} if  for every $j\in J$ the complement $J\setminus\{j\}$ is a digital arc, and $\vert J\vert\geq4$. For a detailed exposition of this theorem, see  \cite{Khalimsky3}.

\begin{theorem}\label{teo:curvadejordanplanodigital}
Let $J$ be a Jordan curve in the digital plane. Then $\mathbb{Z}^{2}\setminus J$ has exactly two connected components.
\end{theorem}

\subsection{Tilings}\label{sec:teselacionesdelplano}
To finish this section, let us recall some basic notions about tilings.

\begin{definition}
Let $\mathcal{T}=\{T_{n}\}_{n\in\mathbb{N}}$ be a countable family of closed, connected subsets of $\mathbb{R}^{2}$. We say that $\mathcal T$ is a \textit{tiling} of $\mathbb R^2$ if:
\begin{enumerate}[\rm(a)]
    \item $\bigcup_{n\in\mathbb{N}}T_{n}=\mathbb{R}^{2}$,
    \item if $i\neq j$, then $\Int{(T_{i})}\cap\Int{(T_{j})}=\emptyset$.
\end{enumerate}
In this case every element $T_n\in\mathcal T$ is called a \textit{tile}.
\end{definition}

Given a tiling $\mathcal{T}=\{T_{n}\}_{n\in\mathbb{N}}$ of $\mathbb R^2$, the points in the plane that are elements of three or more tiles are called \textit{vertices}. If $i,j\in\mathbb{N}$, $i\neq j$, satisfy that $$E_{i,j}:=\{u\in\mathbb{R}^{2}\mid u\in T_{i}\cap T_{j},\;u \text{ is not a vertex}\}\neq\emptyset,$$ then each connected component of the set $E_{i,j}$ is called an \textit{edge}. 
 Lastly, the interior of a tile is called a \textit{face}. If a vertex, an edge or a face are subsets of a tile $T$, we say that it is a vertex, edge or face of the tile, respectively. 

As we mentioned in the introduction, for the purposes of this work, we limit our attention to a particular class of tilings, namely tilings that, besides (a) and (b) of the preceding definition, satisfy the following conditions:
\textit{
\begin{enumerate}[\rm(a)]
    \setcounter{enumi}{2}
    \item $\mathcal{T}$ is a locally finite collection (namely, each point of the plane has a neighborhood that only intersects a finite number of elements of $\mathcal{T}$),
    \item for each $n\in\mathbb{N}$, $T_{n}$ is homeomorphic to the closed unit disk,
    \item each edge of the tiling is homeomorphic to the interval $(0,1)$,
     \item if two different tiles intersect each other, then their intersection is the disjoint union of a finite set of vertices and edges.
\end{enumerate}
}

Consequently, from this point onwards, when we talk about tilings, we will refer to tilings that satisfy these requirements. 

\begin{example}\label{ej:teselacionplano}
The family $\big\{[2n,2n+2]\times[2m,2m+2]\hspace{.2cm}\big\vert\hspace{.2cm}(n,m)\in\mathbb{Z}^{2}\big\}$ is a tiling of the plane. 
\end{example}

In the following proposition we encapsulate basic properties of tilings that can be easily proved from the definition.

\begin{proposition}\label{p:basic tilings}
Let $\mathcal{T}$ be a tiling. 
\begin{enumerate}[\rm(1)]
    \item If $S$ and $T$ are two tiles such that $S\cap T\neq\emptyset$, then $S\cap T=\partial S\cap\partial T$.
    \item Every tile only intersects a finite number of tiles.
    \item Every tile has at least two vertices.
    \item Every tile has the same number of vertices and edges. Moreover, the boundary of a tile is the disjoint union of alternating vertices and edges, and the boundary of every edge is the set of  two vertices that surround it. 
\end{enumerate}
\end{proposition}

From Proposition~\ref{p:basic tilings}-(4), we infer that each edge is determined by two vertices. From its definition, it is clear that each edge is also determined by exactly two faces. Similarly, each face is determined by its boundary, that is, by the vertices and edges of the tile to whom it serves as interior. Also, a vertex is determined both by the faces of the tiles that intersect in it and by the edges that converge in it. In this way we can talk about the faces and edges of a vertex, the vertices and edges of a face, and the faces and vertices of an edge. We denote the sets of faces, edges and vertices of a tiling by $\mathcal{F}_{\mathcal{T}}$, $\mathcal{E}_{\mathcal{T}}$ and $\mathcal{V}_{\mathcal{T}}$, respectively. Lastly, if $v$ is a vertex, we denote the sets of edges and faces of $v$ by $\mathcal{E}_{v}$ and $\mathcal{F}_{v}$, respectively. We can establish analogous notation for the vertices and edges of a face and for the faces and vertices of an edge.

\section{Digital Topology of a Tiling}\label{sec_ Digital Topology of a tiling}

Let $\mathcal{T}$ be a tiling of $\mathbb R^2$. Consider the set  set $\mathcal{D}_{\mathcal{T}}:=\mathcal{F}_{\mathcal{T}}\cup\mathcal{E}_{\mathcal{T}}\cup\mathcal{V}_{\mathcal{T}}$ and the surjection $q:\mathbb R^2\to \mathcal D_{\mathcal T}$ given by
\begin{equation}\label{eq:quotient map}
    q(x)=\begin{cases}
      x,&\text{ if }x\in\mathcal V_{\mathcal T},\\
      E,&\text{ if }x\in E\in\mathcal E_{\mathcal T},\\
      F,&\text{ if }x\in F\in\mathcal F_{\mathcal T}.
    \end{cases}
\end{equation}
Namely, $q$  identifies the points in the same vertex, edge or face.  

\begin{definition}
The set $\mathcal D_{\mathcal T}$ equipped with  the quotient topology induced by $q$ is called the \textit{digital version of $\mathcal{T}$}.
\end{definition}

Let $\mathcal{T}$ be a tiling of $\mathbb R^2$. Clearly we have that the closure (in $\mathbb R^2$) of any vertex of the tiling is the vertex itself. The closure of an edge is the union of the edge and its two vertices. And, the closure of a face is the union of the the face and all  its vertices and edges. This allow us  to describe the closure of any element $x$ in $\mathcal{D}_{\mathcal{T}}$ as follows:

\begin{equation}\label{eq:cerradura} 
\overline{\{x\}} = 
     \begin{cases}
       \{x\} &\quad\text{if $x\in\mathcal{V}_{\mathcal{T}}$,}\\
       \{x\}\cup\mathcal{V}_{x} &\quad\text{if $x\in\mathcal{E}_{\mathcal{T}}$,} \\
       \{x\}\cup\mathcal{V}_{x}\cup\mathcal{E}_{x} &\quad\text{if $x\in\mathcal{F}_{\mathcal{T}}$.}\\
     \end{cases}
\end{equation}

In the following theorem we prove that $\mathcal{D}_{\mathcal{T}}$ is an Alexandrov space, by describing the smallest neighborhood of every element of $\mathcal{D}_{\mathcal{T}}$.

\begin{theorem}\label{teo:teselacionAlexandrov}
Let $\mathcal{T}$ be a tiling. Then $\mathcal{D}_{\mathcal{T}}$ is an Alexandrov space.
Moreover, for every element $x\in\mathcal D_{\mathcal T}$ the smallest neigborhood of $x$ is given by
\[   
N(x) = 
     \begin{cases}
       \{x\}\cup\mathcal{E}_{x}\cup\mathcal{F}_{x}&\quad\text{if $x\in\mathcal{V}_{\mathcal{T}}$,}\\
       \{x\}\cup\mathcal{F}_{x} &\quad\text{if $x\in\mathcal{E}_{\mathcal{T}}$,} \\
       \{x\} &\quad\text{if $x\in\mathcal{F}_{\mathcal{T}}$.}\\
     \end{cases}
\]
\end{theorem}
\begin{proof}
Let $\mathcal{T}$ be a tiling, $q:\mathbb{R}^2\rightarrow\mathcal{D}_{\mathcal{T}}$ the quotient map associated to $\mathcal{D}_{\mathcal{T}}$ and $x\in\mathcal{D}_{\mathcal{T}}$ an arbitrary element.

If $x\in\mathcal{V}_{\mathcal{T}}$ and $U$ is an open neighborhoof of $x$, then $q^{-1}(U)$ is an open, saturated neighborhood of $q^{-1}(x)$. Since $x$ is a vertex, it is contained in the closure of its edges and faces, therefore $q^{-1}(U)$ intersects each element of $\mathcal E_x\cup\mathcal F_x$. Moreover, since $q^{-1}(U)$ is a saturated neighborhood then $q^{-1}(U)$ contains all of the faces and edges of $x$, that is, $q^{-1}\left(\{x\}\cup\mathcal{E}_{x}\cup\mathcal{F}_{x}\right)\subset q^{-1}(U)$. Next we see that $q^{-1}\left(\{x\}\cup\mathcal{E}_{x}\cup\mathcal{F}_{x}\right)$ is an open set. Let $y\in q^{-1}\left(\{x\}\cup\mathcal{E}_{x}\cup\mathcal{F}_{x}\right)$. If $y$ is an element of a face of $x$, then that face is an open set contained in $q^{-1}\left(\{x\}\cup\mathcal{E}_{x}\cup\mathcal{F}_{x}\right)$. If $y$ is an element of an edge $E\in \mathcal E_x$,  then there are two tiles $S, T\in\mathcal T$ that have $x$ as a common vertex and such that $y\in E\subset S\cap T$. Furthermore, for any  other tile $R\in\mathcal{T}\setminus\{S,T\}$, the point $y$ does not belong to $R$. Since $\mathcal{T}\setminus\{T,S\}$ is a locally finite collection of closed sets, its union is a closed subset of $\mathbb{R}^{2}$ (\cite[Theorem 1.1.11.]{Engelking}), thus

$$\mathbb{R}^2\setminus\left(\bigcup_{\substack{R\in\mathcal T\\ R\notin\{S,T\}}}R\right)$$ is an open neighborhood of $y$ contained in $q^{-1}(\{x\}\cup\mathcal{E}_{x}\cup\mathcal{F}_{x})$. Finally, if $y=q^{-1}(x)$ we consider the set $\mathcal{S}$ of tiles that contain $y$, and we recall that $\mathcal{S}$ is a finite set. By the previous argument $\mathbb{R}^2\setminus\left(\bigcup_{T\in(\mathcal{T}\setminus\mathcal{S})}T\right)$ is an open neighborhood of $y$ that is contained in $q^{-1}\left(\{x\}\cup\mathcal{E}_{x}\cup\mathcal{F}_{x}\right)$. This proves that $q^{-1}\left(\{x\}\cup\mathcal{E}_{x}\cup\mathcal{F}_{x}\right)$ is an open set such that $$q^{-1}\left(\{x\}\cup\mathcal{E}_{x}\cup\mathcal{F}_{x}\right)\subset q^{-1}(U).$$ Since $U$ was an arbitrary open neighborhood of $x$ in $\mathcal{D}_{\mathcal{T}}$ we conclude that $\left(\{x\}\cup\mathcal{E}_{x}\cup\mathcal{F}_{x}\right)$ is the smallest neighborhood of $x$ in $\mathcal{D}_{\mathcal{T}}$.

If $x\in\mathcal{E}_{\mathcal{T}}$ and $U$ is an open neighborhood of $x$, then $q^{-1}(U)$ is an open, saturated neighborhood of $q^{-1}(x)$. Since $x$ is an edge we know that it is contained in the closure of its faces, so $q^{-1}(U)$ intersects  the two elements of $\mathcal F_{x}=\{F_1,F_2\}$. Moreover, since $q^{-1}(U)$ is a saturated neighborhood then $q^{-1}(U)$ contains $F_1\cup F_2$. Thus, $q^{-1}\left(\{x\}\cup\mathcal{F}_{x}\right)\subset q^{-1}(U)$. In the same vein as the previous case, we can see that $q^{-1}\left(\{x\}\cup\mathcal{F}_{x}\right)$ is open in $\mathbb{R}^{2}$. Since $U$ is an arbitrary open neighborhood of $x$ in $\mathcal{D}_{\mathcal{T}}$, we have that $\left(\{x\}\cup\mathcal{F}_{x}\right)$ is the smallest neighborhood of $x$ in $\mathcal{D}_{\mathcal{T}}$.

In order to complete the proof, we see that if $x\in\mathcal{F}_{\mathcal{T}}$, then $\{x\}$ is open in $\mathcal{D}_{\mathcal{T}}$ since $x$ is the interior of a tile in $\mathbb{R}^{2}$.  
\end{proof}

We infer from theorems~\ref{teo:resumenconexidad} and \ref{teo:teselacionAlexandrov}, in combination with equality~\ref{eq:cerradura}, that the  adjacency set of any element $x\in \mathcal{D}_{\mathcal{T}}$ is given by

\[   
\mathscr{A}(x) = 
     \begin{cases}
       \mathcal{E}_{x}\cup\mathcal{F}_{x}&\quad\text{if $x\in\mathcal{V}_{\mathcal{T}}$,}\\
       \mathcal{V}_{x}\cup\mathcal{F}_{x} &\quad\text{if $x\in\mathcal{E}_{\mathcal{T}}$,} \\
        \mathcal{V}_{x}\cup\mathcal{E}_{x} &\quad\text{if $x\in\mathcal{F}_{\mathcal{T}}$.}\\
     \end{cases}
\]

This information allows us to describe the connectedness graph of any digital version of a tiling. In Figure \ref{fig:teselacionhexagonal} we illustrate a portion of the tiling of the plane with regular hexagons of the same size as well as the connectedness graph of its digital version. In this graph we depict with white points the vertices of the graph corresponding to the faces of the tiling, with black points the vertices of the graph corresponding to vertices of the tiling, and with black rectangles the vertices of the graph corresponding to the edges of the tiling.

\begin{figure}[htp]
    \centering
    \includegraphics[width=12cm]{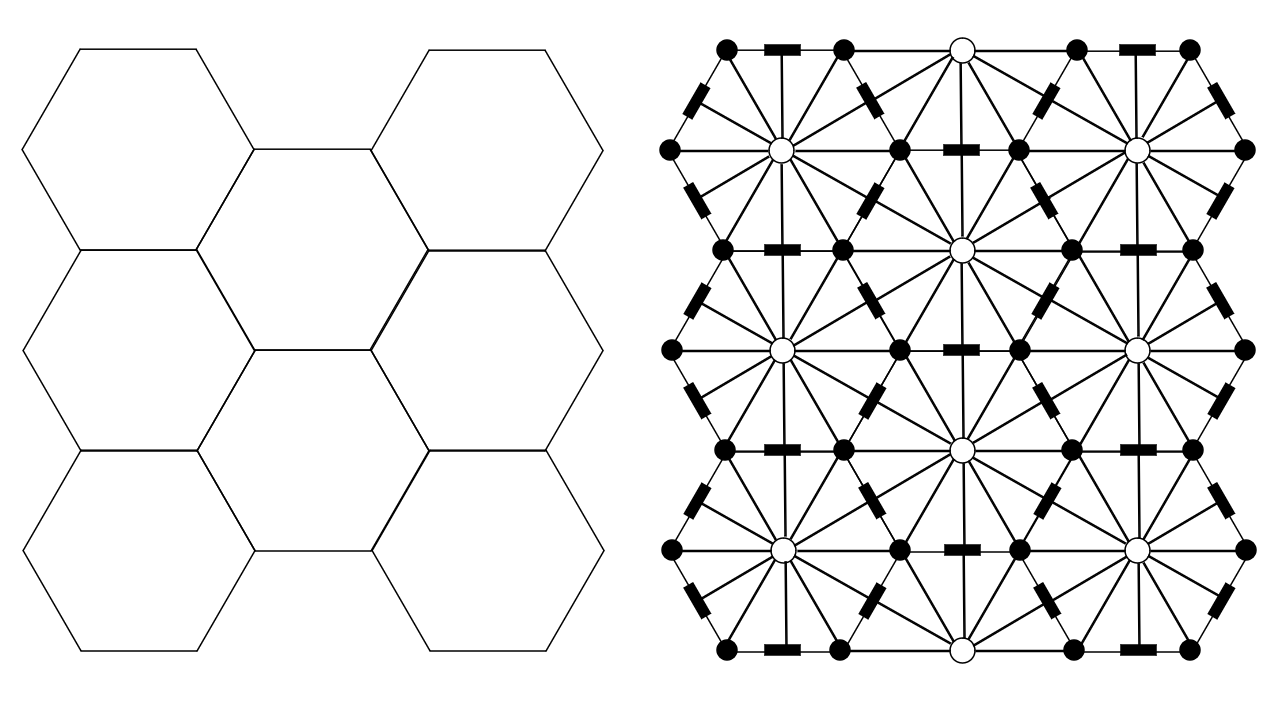}
    \caption{Tiling of the plane with regular hexagons and the connectedness graph of its digital version.}
    \label{fig:teselacionhexagonal}
\end{figure}

In Figure \ref{fig:teselaciondosfiguras} we present another example, illustrating a portion of a tiling of the plane with squares and hexagons, and the connectedness graph of its digital version. The vertices of the graph correspond to the elements of the digital version in the same way as in Figure \ref{fig:teselacionhexagonal}.

\begin{figure}[htp]
    \centering
    \includegraphics[width=11cm]{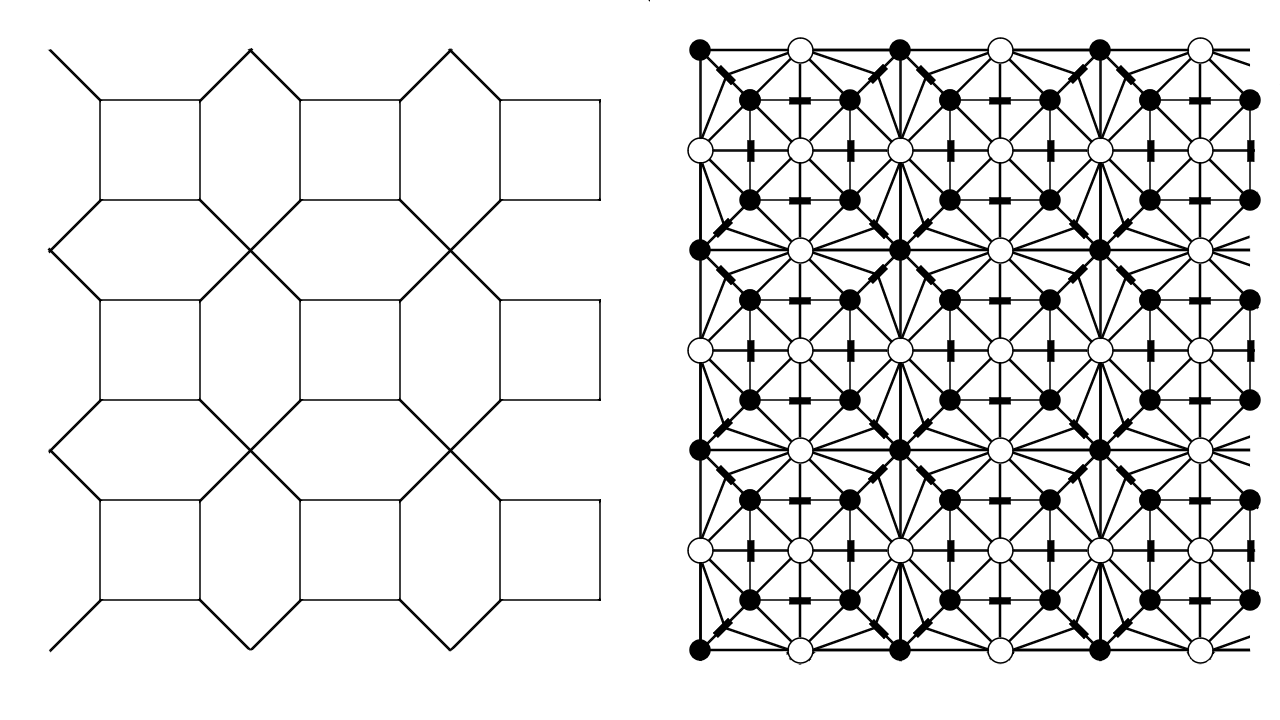}
    \caption{Tiling of the plane with hexagons and squares and the connectedness graph of its digital version.}
    \label{fig:teselaciondosfiguras}
\end{figure}

For the subsequent material, we recall the definition of a planar graph:

\begin{definition}\label{def:graficaplana}
Let $G=(V_G, E_G)$ be a graph. We say that $G$ is a planar graph if there is a function $\psi:V_{G}\cup E_{G}\rightarrow\mathscr{P}(\mathbb{R}^{2})$ such that:
\begin{enumerate}[\rm(1)]
    \item The image of each element of $V_{G}$ is a singleton.
    \item The image of each element of $E_{G}$ is a subset of $\mathbb{R}^{2}$ homeomorphic to $(0,1)$.
    \item For any two elements $x,y\in V_{G}\cup E_{G}$, $x\neq y$, $$\psi(x)\cap\psi(y)=\emptyset.$$
    \item If $v$ and $w$ are the vertices of the edge $e$, then $$\overline{\psi(e)}=\psi(e)\cup\psi(v)\cup\psi(w).$$ 
\end{enumerate} 

In this situation, we say that $\vert G\vert:=\psi(V_{G}\cup E_{G})$ is a planar embedding of $G$. Sometimes the set $|G|$ is also called a \textit{geometric realization} of $G$.

\end{definition}

\begin{lemma}\label{lema:graficaplana}
Let $\mathcal{T}$ be a tiling. The connectedness graph of $\mathcal{V}_{\mathcal{T}}\cup\mathcal{E}_{\mathcal{T}}$, with the subspace topology induced by $\mathcal{D}_{\mathcal{T}}$, is a planar graph.
\end{lemma} 

\begin{proof}
Let $H=(V_H, E_H)$ be the connectedness graph of $\mathcal{V}_{\mathcal{T}}\cup\mathcal{E}_{\mathcal{T}}$. The set of vertices of $H$ is precisely the set $V_{H}=\mathcal{V}_{\mathcal{T}}\cup\mathcal{E}_{\mathcal{T}}$. We know that the edges of $H$ always denote the adjacency between an edge of the tiling and its two vertices. Therefore, every edge in $E_{H}$ is a pair of the form $\{A,z\}$, where $A\in \mathcal E_z\subset \mathcal E_{\mathcal T}$ and $z\in\mathcal V_A\subset \mathcal V_{\mathcal T}$.
For every $A\in\mathcal{E}_{\mathcal{T}}$, pick a point $x_{A}\in A\subset\mathbb R^2$. If $x\in\mathcal{V}_{\mathcal{T}}$ is a vertex of $A\in\mathcal{E}_{\mathcal{T}}$, we define $I_{x_{A},x}$ as the connected component of $A\setminus\{x_{A}\}$ delimited by $x_{A}$ and $x$. It is clear that $\{x_{A},x\}\cap I_{x_{A}, x}=\emptyset$. Since $A$ is an edge of the tiling, then $I_{x_{A},x}$ is homeomorphic to the interval $(0,1)$ and $$\overline{I_{x_{A},x}}=I_{x_{A},x}\cup\{x_{A},x\}.$$

Moreover, if $\{x_{A},x\}\neq\{x_{B},y\}$ then $I_{x_{A},x}\cap I_{x_{B},y}=\emptyset$, considering that any two different edges of the tiling have an empty intersection. Therefore the function $\psi:V_{H}\cup E_{H}\rightarrow\mathscr{P}(\mathbb{R}^{2})$ given by

\[   
\psi(x) = 
     \begin{cases}
       \{x\} &\quad\text{if}\hspace{.2cm} x\in\mathcal{V}_{\mathcal{T}},\\
       \{x_{A}\} &\quad\text{if}\hspace{.2cm} x=A\hspace{.2cm}\text{for some}\hspace{.2cm}A\in\mathcal{E}_{\mathcal{T}}, \\
       I_{x_{A},z} &\quad\text{if}\hspace{.2cm} x=\{A,z\}\in E_{H},\hspace{.2cm}\text{where}\hspace{.2cm}A\in\mathcal{E}_{\mathcal{T}}\hspace{.2cm}\text{and}\hspace{.2cm}z\in\mathcal{V}_{\mathcal{T}},
     \end{cases}
\]

proves the planarity of $H$.
\end{proof}

\begin{proposition}\label{prop:graficaplana}
Let $\mathcal{T}$ be a tiling. The connectedness graph of $\mathcal{D}_{\mathcal{T}}$ is a planar graph.
\end{proposition}
\begin{proof}

Let $G=(V_{G},E_{G})$ be the connectedness graph of $\mathcal{D}_{\mathcal{T}}$ and $H=(V_H, E_H)$ the connectedness graph of $\mathcal{V}_{\mathcal{T}}\cup\mathcal{E}_{\mathcal{T}}$. Then $H$ is a subgraph of $G$ and by Lemma \ref{lema:graficaplana} there is a bijection $\psi:V_{H}\cup E_{H}\rightarrow\vert H\vert\subset \mathscr{P}(\mathbb{R}^{2})$ from the set of vertices and edges of $H$ to a planar embedding $\vert H\vert$ of $H$. In order to prove the planarity of $G$, we want to extend the definition of $\psi$ to a function $\phi:V_{G}\cup E_{G}\rightarrow\mathbb R^2$ satisfying conditions (1)-(4) of Definition~\ref{def:graficaplana}. For that purpose, it suffices to define $\phi$ for every element of  $\mathcal{F}_{\mathcal{T}}$ and for every edge of $G$ connecting a face of the tiling with one of its vertices or edges.

Let $C\in\mathcal{F}_{\mathcal{T}}$, we know that there is a unique tile $T$ that has $C$ as a face, and we know that $\mathscr{A}(C)=\mathcal{V}_{C}\cup\mathcal{E}_{C}$. Let $\gamma:\mathbb{B}^{2}\rightarrow T$ be a homeomorphism between $\mathbb{B}^2$, the unit closed disk of $\mathbb R^2$, and $T$. Consider the set $\gamma^{-1}(\mathcal{V}_{C})=\{w_1, w_2, \dots, w_n\}$, which is contained in $\mathbb{S}^{1}$ since $\gamma$ is a homeomorphism and $\mathcal{V}_{C}\subset\partial T$. Without lost of generality we can suppose that $w_{1}, \dots, w_{n}$ are numbered clockwise. Then we can number the set $\gamma^{-1}(\mathcal{E}_{C})=\{B_{1}, \dots,  B_{n}\}$ denoting by $B_{1}$ the arc delimited by $w_{1}$ and $w_{2}$ and then proceeding in a clockwise manner. For every  $i\in\{1, \dots,n\}$, define $A_{i}:=\gamma(B_{i})$, considering that $\gamma$ is a homeomorphism we can assure that $\mathcal{E}_{C}=\{A_{1}, \dots, A_{n}\}$. Being consistent with the notation introduced in Lemma \ref{lema:graficaplana}, we have a point $x_{A_{i}}\in A_{i}$ such that $\psi(A_{i})=\{x_{A_{i}}\}$. Let $b_{i}:=\gamma^{-1}(x_{A_{i}})$, which is an element of $B_i\subset\mathbb{S}^{1}\setminus(\gamma^{-1}(\mathcal{V}_{C}))$.

For every $w\in\{w_{1}, \dots,w_{n}, b_{1}, \dots, b_{n}\}$, let $L_{w}:=\left\{t w~\vert~t\in(0,1)\right\}.$ Now we consider the family  $$\mathcal{L}_{C}:=\left\{L_{w}~\vert~w\in\{w_{1}, \dots,w_{n}, b_{1}, \dots, b_{n}\}\right\},$$
whose elements are pairwise disjoint and they are contained in $\Int{\mathbb{B}^{2}}$. 

Let $x_{C}:=\gamma(0,0)$. Since $(0,0)\in\Int{\mathbb{B}^{2}}$, it follows that $x_{C}\in C$. Each element of the family $\gamma(\mathcal{L}_{C})=\{\gamma(L)~\vert~ L\in\mathcal{L}_{C}\}$ is an  arc delimited by $x_{C}$ and one  of the points in the set $\{v_{1}, \dots,v_{n}, x_{A_{1}}, \dots, x_{A_{n}}\}$, where $v_i=\gamma(w_i)\in\mathcal V_C$. Thus we can denote each arc $\gamma(L)\in\gamma(\mathcal{L}_{C})$ as $M_{x_{C}, v}$, where $v$ is the point in $\{v_{1}, \dots,v_{n}, x_{A_{1}}, \dots, x_{A_{n}}\}$ that, alongside $x_{C}$, delimits $\gamma(L)$. From its definition, it is clear that $M_{x_{C}, v}\cap\{x_{C},v\}=\emptyset$.
These arcs are pairwise disjoint, since $\gamma$ is a homeomorphism and the elements of $\mathcal{L}_{C}$ are pairwise disjoint. Also,  each arc $M_{x_C,v}$ is contained in $C$, since the elements of $\mathcal{L}_{C}$ are contained in $\Int{\mathbb{B}^{2}}$. From here we can infer that  if $D$ is another face of the tiling, then the elements of $\gamma(\mathcal{L}_{C})$ and $\gamma(\mathcal{L}_{D})$ are pairwise disjoint. Furthermore, we can also conclude that the elements of $\gamma(\mathcal{L}_{C})$ do not intersect $\partial T$. Lastly, note that $M_{x_{C}, v}\in\gamma(\mathcal{L}_{C})$ implies $$\overline{M_{x_{C}, v}}=M_{x_{C}, v}\cup\{x_{C},v\}.$$

Therefore the function $\phi:V_{G}\cup E_{G}\rightarrow\mathscr{P}(\mathbb{R}^{2})$ given by

\[   
\phi(x) = 
     \begin{cases}
      \psi(x) &\quad\text{if}\hspace{.2cm} x\in V_{H}\cup E_{H},\\
       \{x_{C}\} &\quad\text{if}\hspace{.2cm} x=C\hspace{.2cm}\text{for some}\hspace{.2cm}C\in\mathcal{F}_{\mathcal{T}}, \\
       M_{x_{C},v} &\quad\text{if}\hspace{.2cm} x=\{C,v\}\in E_{G}\setminus E_{H},\hspace{.2cm}\text{where}\hspace{.2cm}C\in\mathcal{F}_{\mathcal{T}}\hspace{.2cm}\text{and}\hspace{.2cm}v\in\mathcal{V}_{C}, \\
       M_{x_{C}, x_{A}} &\quad\text{if}\hspace{.2cm} x=\{C,A\}\in E_{G}\setminus E_{H},\hspace{.2cm}\text{where}\hspace{.2cm}C\in\mathcal{F}_{\mathcal{T}}\hspace{.2cm}\text{and}\hspace{.2cm}A\in\mathcal{E}_{C}, \\
     \end{cases}
\]

proves the planarity of $G$.
\end{proof}

\section{A Jordan Curve Theorem for Tilings}\label{Sec:main theorem}

In this section we present our generalization of Khalimsky's Jordan curve theorem. To achieve this we rely on the connectedness graph of the digital version of a tiling.

First we generalize the definition of a digital Jordan curve that we have previously introduced in Section~\ref{sec:preliminares}.
\begin{definition}
Let $\mathcal{T}$ be a tiling. A subset $J$ of $\mathcal{D}_{\mathcal{T}}$ is a \textit{digital Jordan curve} if $\vert J\vert\geq4$ and if for every $j\in J$, $J\setminus\{j\}$ is a digital arc. 
\end{definition}

Recall that $J\setminus\{j\}$ is a digital arc in $\mathcal{D}_{\mathcal{T}}$ if and only if $J\setminus\{j\}$ is an induced simple path in the connectedness graph of $\mathcal{D}_{\mathcal{T}}$ (Remark \ref{obs:arcocamino}). This correspondence allows us to define a Jordan curve in the context of graph theory.

\begin{definition}
Let $G$ be a graph. A graph-theoretical Jordan curve is an induced cycle of length equal or greater than four.
\end{definition}

Let $G$ be a graph and $x$ a vertex of $G$. Keeping in mind the correspondence between a topological space and its connectedness graph, we denote by $\mathscr{A}(x,G)$ the subgraph induced by the vertices of $G$ adjacent to $x$. We say that $G$ is \textit{locally Hamiltonian} if for every vertex $x$ of $G$, $\mathscr{A}(x,G)$ has a Hamiltonian cycle. The following theorem is crucial for our generalization and it is due to V. Neumann-Lara and R. Wilson (\cite{Neumann-Lara}). 

\begin{theorem}\label{teo:jordangraficas}
If $G$ is a locally Hamiltonian, connected graph and $J\subset G$ is a graph-theoretical Jordan curve, then
\begin{enumerate}[\rm a)]
    \item the complement of $J$ has at most two connected components in $G$, 
    \item if $G$ is planar, then the complement of $J$ has exactly two connected components in $G$.
\end{enumerate}
\end{theorem}

It can be seen from Neumann-Lara and Wilson's proof that if $G$ is an infinite, planar, locally Hamiltonian and connected graph and $J\subset G$ is a graph-theoretical Jordan curve, then out of the two connected components that make up the complement of $J$ in $G$, one is bounded whereas the other one is not. We call the set of vertices of the bounded component the \textit{interior} of $V_G\setminus V_J$, and denote it by $I(J)$. In a similar manner, the set of vertices of the unbounded component, $O(J)$, are the \textit{exterior} of $V_{G}\setminus V_{J}$.

\begin{theorem}\label{teo:jordanteselaciones}
Let $\mathcal{T}$ be a tiling and $J\subset\mathcal{D}_{\mathcal{T}}$ a digital Jordan curve. Then $\mathcal{D}_{\mathcal{T}}\setminus J$ has exactly two connected components.
\end{theorem}

\begin{proof}

Let $G$ be the connectedness graph of $\mathcal{D}_{\mathcal{T}}$, it suffices to see that $G$ satisfies the hypotheses of Theorem \ref{teo:jordangraficas}. Let $q:\mathbb{R}^{2}\rightarrow \mathcal{D}_{\mathcal{T}}$ be the quotient function associated to the topology of $\mathcal{D}_{\mathcal{T}}$ (see equation~\ref{eq:quotient map}). We know that $q$ is continuous and therefore the connectedness of $\mathcal{D}_{T}$ is guaranteed by the connectedness of $\mathbb{R}^{2}$. Moreover, we know that $\mathcal{D}_{T}$ is an Alexandrov discrete space, so by  Theorem~\ref{teo:resumenconexidad} we have that $G$ is connected. In addition,   $G$ is a planar graph by Proposition \ref{prop:graficaplana}.

It only rests us to prove that $G$ is a locally Hamiltonian graph. By Theorem \ref{teo:teselacionAlexandrov}, if  $x$ is a vertex of $G$ we have that:

\[   
\mathscr{A}(x,G) = 
     \begin{cases}
       \mathcal{E}_{x}\cup\mathcal{F}_{x}&\quad\text{if $x\in\mathcal{V}_{\mathcal{T}}$,}\\
       \mathcal{V}_{x}\cup\mathcal{F}_{x} &\quad\text{if $x\in\mathcal{E}_{\mathcal{T}}$,} \\
        \mathcal{V}_{x}\cup\mathcal{E}_{x} &\quad\text{if $x\in\mathcal{F}_{\mathcal{T}}$.}\\
     \end{cases}
\]

If $x\in\mathcal{V}_{\mathcal{T}}$, we know that at least three tiles intersect in $x$. Each of these tiles has exactly two edges whose vertex is $x$. This fact allows us to infer that if we consider the set $\mathcal E_{x}$ (which is finite) and we number their elements clockwise, then every two consecutive edges determine a face of $x$ (and none of these faces is determined by two different pairs of edges). We also know that every edge of $x$ is determined by exactly two faces of $x$. Lastly, we know that in the connectedness graph there are no two adjacent edges nor two adjacent faces. Thus the subgraph induced by $\mathcal{E}_{x}\cup\mathcal{F}_{x}$ is a cycle wherein the vertices of the graph corresponding to edges of the tiling and the vertices of the graph corresponding to faces of the tiling alternate.

If $x\in\mathcal{E}_{\mathcal{T}}$, $x$ has exactly two vertices and two faces, furthermore, both vertices are vertices of both faces. For that reason the subgraph of $G$ induced by $\mathcal{V}_{x}\cup\mathcal{F}_{x}$ is a cycle with exactly four vertices. 

Finally, if $x\in\mathcal{F}_{\mathcal{T}}$ we recall the analysis made in Proposition \ref{prop:graficaplana} to conclude that the vertices and edges of $x$ alternate in $\partial x$, allowing us to verify that $\mathcal{V}_{x}\cup\mathcal{E}_{x}$ is a cycle.

In any case, we have proved that the subgraph induced by $\mathscr{A}(x,G)$ is a cycle and thus $G$ is locally Hamiltonian. In addition, we know that since $J$ is a digital Jordan curve, then the subgraph induced by $J$ in $G$ is a graph-theoretical Jordan curve. Hence, by  Theorem \ref{teo:jordangraficas}  we can conclude that $G\setminus J$ has exactly two connected components in $G$.

To finish the proof, recall that $\mathcal{D}_{\mathcal{T}}$ is an Alexandrov space. Since $G\setminus J$ has exactly two connected components in $G$, by Theorem~\ref{teo:resumenconexidad}  we have that $\mathcal{D}_{\mathcal{T}}\setminus J$ has exactly two connected components. Moreover, we know that the connected components of $\mathcal{D}_{\mathcal{T}}\setminus J$ are precisely the vertices of the connected components of $G\setminus J$; in other words, the connected components of $\mathcal{D}_{\mathcal{T}}\setminus J$ are its interior $I(J)$, which is bounded, and its exterior $O(J)$, which is not.
\end{proof}

\section{Open and Closed Digital Jordan Curves}\label{sec:open and closed}

 In \cite{Khalimsky4},  Khalimsky, Kopperman and Meyer developed a theory on the processing of digital pictures by means of studying boundaries. Their work relies on properties of the digital plane and on Khalimsky's Jordan curve theorem. The digital plane is not a homogeneous space, nonetheless for any two faces $C_{1}, C_{2}\in\mathcal{F}_{\mathcal{T}}$ (where $\mathcal{T}$ is the tiling of Example \ref{ej:teselacionplano}) there is a homeomorphism from the digital plane into itself that maps $C_1$ to $C_2$. For this reason, it is convenient to think about the pixels on a screen as the faces of the tiling.

Motivated by this approach, we are now interested in studying digital Jordan curves $J$ that satisfy two conditions: 
\begin{itemize}
 \item[(W1)]half of the elements of $J$ are faces of the tiling (and they alternate with the non-faces elements of the curve).
 \item[(W2)] $J$ encloses a face of the tiling (namely,  $I(J)\cap \mathcal F_{\mathcal T}\neq\emptyset$).
\end{itemize}

\begin{lemma}\label{lema:ultimolema}
Let $\mathcal{T}$ be a tiling and $J\subset\mathcal{D}_{\mathcal{T}}$ a digital Jordan curve. Then, for every $x\in J$, $\mathscr{A}(x)\cap I(J)\neq\emptyset$.
\end{lemma}
\begin{proof}
Let $x\in J$ and let $G$ be the connectedness graph of $\mathcal{D}_{\mathcal{T}}$. We proceed by contradiction. Suppose that $\mathscr{A}(x)\cap I(J)=\emptyset$ and let $x_{0}$ and $x_{1}$ be the two adjacent points to $x$ in $J$. It is clear that $x_{0}$ and $x_{1}$ are not adjacent. We know that the subgraph induced by $\mathscr{A}(x)$ is a cycle with at least four elements. For that reason there are two paths contained in $\mathscr{A}(x,G)$ that only intersect at $x_{0}$ and $x_{1}$. Let $a$ and $b$ be two vertices, different from $x_{0}$ and $x_{1}$, in each of these paths. Since $\mathscr{A}(x)\cap I(J)=\emptyset$, then, with the exception of $x_{0}$ and $x_{1}$, the vertices of the paths we mentioned earlier are contained in $O(J)$. In particular, $a,b\in O(J)$.

Since $\mathcal{T}$ is a tiling, $G$ is a planar graph and thus there is a map $\psi:V_{G}\cup E_G\to\mathscr{P}(\mathbb{R}^{2})$ that attests the planarity of $G$. Let $H$ be the subgraph induced by the vertices $V_{H}:=I(J)\cup J$ in $G$. Being a subgraph of $G$, $H$ is a planar graph too. Since $\mathscr{A}(x)\cap I(J)=\emptyset$, we have that $\mathscr{A}(x)\cap (I(J)\cup J)=\{x_{0}, x_{1}\}$, therefore $x$ is a vertex of exactly two faces of $H$. Let $\mathrm{C}$ be the cycle that determines the bounded face of $H$ having $x$ as a vertex (the other face contains the exterior face of the subgraph induced by the vertices of $J$). Then we know that $\bigcup \psi(\mathrm{C})$ is the boundary of a set $\mathrm{D}$, homeomorphic to a closed disk, and that $\psi(x_0),\psi(x),\psi(x_1)\subset \bigcup \psi(\mathrm{C})=\partial \mathrm{D}$.

Now we construct a graph $\Gamma$ with the set of vertices $V_{\Gamma}=V_{G}\cup\{p\}$, where $p\notin V_{G}$, and the set of edges $E_{\Gamma}=E_{G}\cup\left\{\{p,w\}~\vert~w\in V_{\mathrm{C}}\right\}$. We sustain that $\Gamma$ is a planar graph. Indeed, since $\mathrm{D}$ is homeomorphic to a closed disk and $\psi(x_0),\psi(x),\psi(x_1)\subset\partial\mathrm{D}$, we can replicate the construction of the point $x_{C}$ and the sets $M_{x_{C},v}$ in the proof of Proposition \ref{prop:graficaplana} to prove the existence of a point $x_{p}\in\Int\mathrm{D}$ and sets $M_{x_{p},w}\subset\Int\mathrm{D}$ such that the function $\phi:V_{\Gamma}\cup E_{\Gamma}\rightarrow\mathscr{P}(\mathbb{R}^{2})$ defined as

\[   
\phi(x) = 
     \begin{cases}
      \psi(x) &\quad\text{if}\hspace{.2cm} x\in V_{G}\cup E_{G},\\
       \{x_{p}\} &\quad\text{if}\hspace{.2cm} x=p, \\
       M_{x_{p},w} &\quad\text{if}\hspace{.2cm} x\in E_{\Gamma}\setminus E_{G},\hspace{.2cm}\text{where}\hspace{.2cm}p\hspace{.2cm}\text{and}\hspace{.2cm}w\in V_{\mathrm{C}} \\
       &\quad\text{are the vertices of}\hspace{.2cm}x, \\
     \end{cases}
\]

proves the planarity of $\Gamma$.

However, if we take a look at the sets of vertices $\{p,a,b\}$ and $\{x_0,x,x_1\}$, we see that $\Gamma$ contains a subdivision of a complete bipartite graph $K_{3,3}$, which contradicts Kuratowski's theorem (\cite[Theorem 4.4.6]{Diestel}). Therefore we have proved that $\mathscr{A}(x)\cap I(J)\neq\emptyset$.
\end{proof}

\begin{proposition}\label{prop:ultimaprop}
Let $\mathcal{T}$ be a tiling and $J\subset\mathcal{D}_{\mathcal{T}}$ a digital Jordan curve. The following are equivalent:
\begin{enumerate}[\rm a)]
\item $J$ is closed.
\item $I(J)$ and $O(J)$ are open.
\item $J$ is nowhere dense.
\item $I(J)\cup O(J)$ is dense.
\item $J\cap\mathcal{F}_{\mathcal{T}}=\emptyset$.
\item $J=\partial(I(J))$.
\end{enumerate}

\end{proposition}

\begin{proof}

a)$\implies$b). $I(J)\cup O(J)$ is the complement of $J$ and thus it is open. In addition, since $\mathcal D_{\mathcal T}$  is an Alexandrov space,  Theorem \ref{teo:resumenconexidad} guarantees that $I(J)$ and $O(J)$  are open and closed in $I(J)\cup O(J)$, because they are the connected components of $I(J)\cup O(J)$. Since  the latter is open in $\mathcal D_{\mathcal T}$, we conclude that $I(J)$ and $O(J)$ are also open in $\mathcal{D}_{\mathcal{T}}.$

b)$\implies$c). We prove the contrapositive. Suppose there is an element $x\in\Int{\left(\overline{J}\right)}$.  This implies that $N(x)\subset\overline{J}$. In particular, we have that $x\in\overline{J}$ and hence $\overline{\{x\}}\subset\overline{J}$. Thus $\mathscr{A}(x)\cup\{x\}\subset\overline{J}$. In this situation, $\overline{J}$ can not be a digital Jordan curve since the connectedness graph of $\mathscr{A}(x)\cup\{x\}$ contains at least one vertex, one edge and one face adjacent to each other; in other words, the connectedness graph of $\mathscr{A}(x)\cup\{x\}$ contains cycles of length $3$. For that reason, $J$ is not closed and therefore $I(J)\cup O(J)$ is not open, which in turn implies that \textit{b)} is false.

c)$\implies$d).  Suppose that $I(J)\cup O(J)$ is not dense. Since $\mathcal{D}_{\mathcal{T}}$ is the disjoint union of $I(J)$, $O(J)$ and $J$, there is a point $x\in J$ such that $N(x)\cap\left( I(J)\cup O(J)\right)=\emptyset$. Therefore $N(x)\subset J$, so $x\in\Int(J)$ and thus $\Int{\left(\overline{J}\right)}\neq\emptyset$.

d)$\implies$e). Assume there is a point  $x\in J\cap\mathcal{F}_{\mathcal{T}}$. Thus $N(x)=\{x\}\subset J$. Since $J\cap\left(I(J)\cup O(J)\right)=\emptyset$, we have that $N(x)\cap\left(I(J)\cup O(J)\right)=\emptyset$ and therefore $I(J)\cup O(J)$ cannot be dense.

e)$\implies$f). First let us prove that  $J$ is closed. By e), the elements of $J$ are alternating vertices and edges. The closure of a vertex is the vertex  itself, and the closure of an edge is the union of the  edge and its two vertices. Since the vertices and edges alternate in $J$, the vertices of each edge in $J$ are contained in $J$. Therefore the closure of each element of $J$ is contained in $J$. Since $\mathcal{D}_{\mathcal{T}}$ is an Alexandrov discrete space this implies that $J$ is closed.

Now by the implication a)$\implies$b), we infer that $I(J)$ and $O(J)$ are open. This proves that  $\partial(I(J))=\overline{(I(J))}\setminus I(J)$. Furthermore, since $\mathcal{D}_{\mathcal{T}}$ is the disjoint union of $J$, $I(J)$ and $O(J)$, we conclude that   $I(J)\cup J$ is a closed set that contains $I(J)$. Thus
$$\partial(I(J))=\overline{(I(J))}\setminus I(J)\subset (I(J)\cup J)\setminus I(J)=J.$$

It rests us  to prove that $J\subset\partial (I(J))$. Since $I(J)$ and $J$ are disjoint sets, it is enough to prove that  $N(x)\cap I(J)\neq\emptyset$ for every $x\in J$. Let $x\in J$, by Lemma \ref{lema:ultimolema} we have that $\mathscr{A}(x)\cap I(J)\neq\emptyset$. If $x$ is an edge, then the elements of $N(x)$ are $x$ and its two faces, and  the elements of $\mathscr{A}(x)$ are the two faces and the two vertices of $x$. We know that $J$ and $I(J)$ are disjoint, and that the vertices of $x$ are contained in $J$. Thus $\mathscr{A}(x)\cap I(J)$ contains at least one of the faces of $x$, implying that $N(x)\cap I(J)\neq\emptyset$. Lastly, if $x$ is a vertex, then $\mathscr{A}(x)\subset N(x)$ and therefore $N(x)\cap I(J)\neq\emptyset$. In any case, $J\subset\partial (I(J))$.

f)$\implies$a). This follows from the fact that the boundary of any subset of a topological space is always closed.
\end{proof}

Analogously, we can obtain a dual version of  Proposition \ref{prop:ultimaprop}.  We let the proof to the reader, since it is similar to the previous one. 

\begin{proposition}\label{prop:ultimaultimaprop}
Let $\mathcal{T}$ be a tiling and $J\subset\mathcal{D}_{\mathcal{T}}$ a digital Jordan curve. The following are equivalent:
\begin{enumerate}[\rm(a)]
\item $J$ is open.
\item $I(J)$ and $O(J)$ are closed.
\item $J\cap\mathcal{V}_{\mathcal{T}}=\emptyset$.
\end{enumerate}

\end{proposition}

After Proposition~\ref{prop:ultimaultimaprop}-(c), we conclude that 
open digital curves are made exclusively of faces and edges (hence, they satisfy condition W1).
 Furthermore, if $J$ is an open digital Jordan curve, then $I(J)$ is closed, so for every $x\in I(J)$, $\overline{\{x\}}\subset I(J)$, and whether $x$ is a vertex, an edge or a face, $\overline{\{x\}}$ always contains a vertex of the tiling.
 This reasoning yields the following remark.
 \begin{remark}\label{r:Jopen I contains vertex}
 If $J$ is an open digital curve, then $I(J)\cap \mathcal V_{\mathcal T}\neq \emptyset$.
 \end{remark}

However, we cannot  guarantee that $I(J)$ contains a face of the tiling. This situation is pictured for the tiling of the plane with regular hexagons in Figure \ref{fig:ultimafigura}. In this figure, the dotted line intersects all the elements of a digital Jordan curve consisting only of edges and faces of the tiling, whilst the thick black line intersects all the elements of its interior, none of which is a face of the tiling.

\begin{figure}[htp]
    \centering
    \includegraphics[width=12cm]{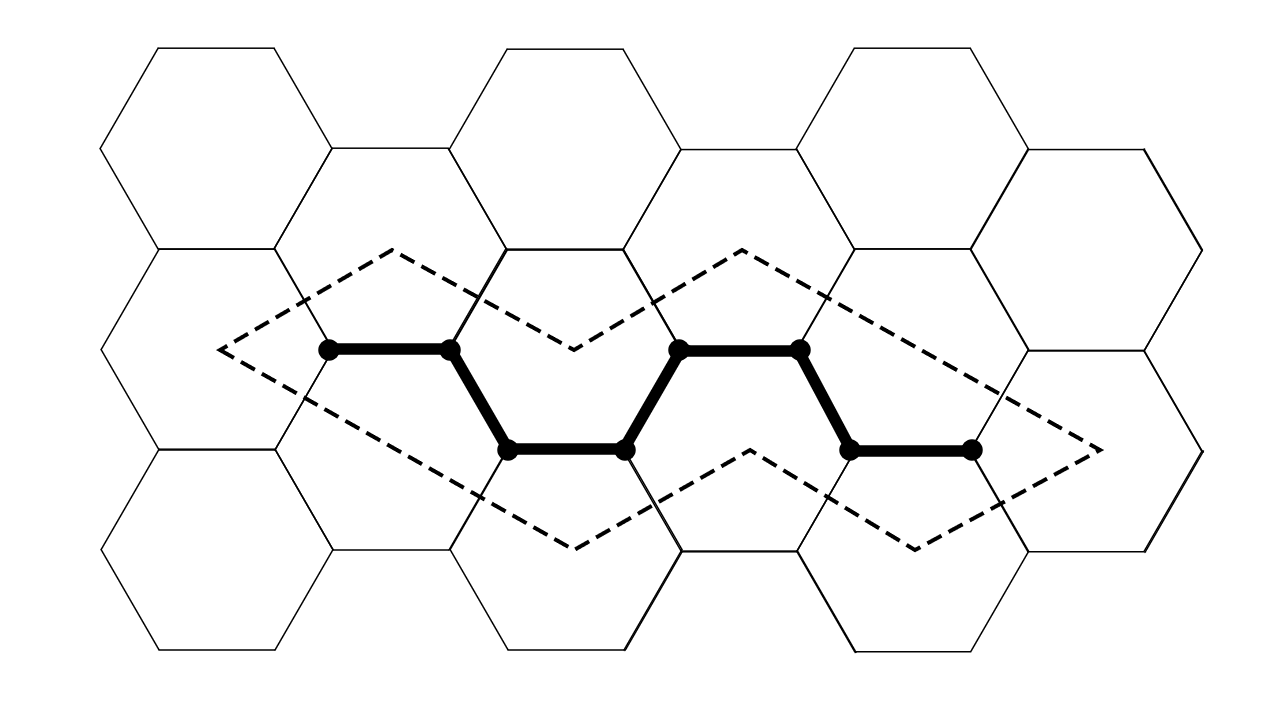}
    \caption{A digital Jordan curve and its interior.}
    \label{fig:ultimafigura}
\end{figure}

On the other hand, by Proposition~\ref{prop:ultimaprop}-(d), every  closed digital Jordan curve  satisfies condition W2. However, this kind of curves may not be very interesting, since they do not contain any face (and therefore they do not satisfy condition W1).
This is why, if we are looking for a class of digital Jordan curves satisfying W1 and W2, we need to add an extra condition. We formalize this in the  following section.

\section{Well-behaved Digital Jordan Curves}\label{Sec:Well behaved}
 Intuitively, the problem of the example in Figure~\ref{fig:ultimafigura} is that the elements of $J$ that are faces of the tiling, are very close to each other. In order to prevent this from happening, we need the following.

\begin{definition}
Let $\mathcal{T}$ be a tiling and $J\subset\mathcal{D}_{\mathcal{T}}$ a digital Jordan curve. We say that $J$ is \textit{well-behaved} if for every $C\in\mathcal F_{\mathcal T}\cap J$ and every $x\in\mathcal V_{C}$, $\mathcal{F}_{x}\not\subset J$.
\end{definition}
Clearly the curve of Figure~\ref{fig:ultimafigura} is not well-behaved.

\begin{lemma}\label{l:wb caras en interior}
\label{obs:ultimaobs}
Let $\mathcal{T}$ be a tiling and $J\subset\mathcal{D}_{\mathcal{T}}$ a well-behaved digital Jordan curve. If $\mathcal{V}_{\mathcal{T}}\cap I(J)\neq\emptyset$ (in particular, if $J$ is a well-behaved open digital jordan curve), then $\mathcal{F}_{\mathcal{T}}\cap I(J)\neq\emptyset$. 
\end{lemma}
\begin{proof}

 Since $I(J)$ and $O(J)$ are the components of $Y:=I(J)\cup O(J)$, we can find an open set $U\subset \mathcal D_{\mathcal T}$ such that
$U\cap Y=I(J)$. Thus, 
$$I(J)\subset U\subset I(J)\cup J.$$

Assume that  $\mathcal F_{\mathcal T}\cap I(J)=\emptyset$, and  pick an element $x\in\mathcal{V}_{\mathcal{T}}\cap I(J)$. Notice that 
$$\mathcal F_x\subset N(x)\subset U\subset I(J)\cup J.$$
Now, we can use the assumption that 
 $\mathcal{F}_{\mathcal{T}}\cap I(J)=\emptyset$ to conclude that   $\mathcal{F}_{x}\subset J$, which is impossible because $J$ is well-behaved. Therefore $\mathcal F_{\mathcal T}\cap I(J)\neq\emptyset$, as desired.
\end{proof}

As an immediate consequence of Lemma~\ref{obs:ultimaobs}, we have the following.
\begin{remark}\label{remark.wellbehaved are nice}
Let $\mathcal{T}$ be a tiling and $J\subset\mathcal{D}_{\mathcal{T}}$ a well-behaved open digital Jordan curve. Then $J$ satisfies condition W1 and W2.
\end{remark}

To finish this paper we will use the theory we have developed in order to prove a generalization of  Rosenfeld's theorem that we presented on the introduction (Theorem~\ref{teo:Rosenfeld}). 

Let $\mathcal{D}_{\mathcal{T}}$ be the digital plane. Recall that Rosenfeld works with a squared grid endowed with the $k$-adjacency relations ($k\in\{4,8\}$) previously described.  If we think of the faces of the digital plane as the points in Rosenfeld's grid, two faces $C_{1}, C_{2}\in\mathcal{F}_{\mathcal{T}}$ are \textit{$4$-adjacent} if $\mathcal{E}_{C_{1}}\cap\mathcal{E}_{C_{2}}\neq\emptyset$. Similarly, we say that they are  \textit{$8$-adjacent} if $\mathcal{V}_{C_{1}}\cap\mathcal{V}_{C_{2}}\neq\emptyset$. 
We also give the corresponding definition of a simple closed $4$-path with at least five points: $J\subset \mathcal{D}_{\mathcal{T}}$ is a \textit{$4$-Jordan curve} if $J$ is an open digital Jordan curve such that $\vert\mathcal{F}_{\mathcal{T}}\cap J\vert\geq5$ and for every $C\in\mathcal{F}_{\mathcal{T}}\cap J$, $C$ has exactly two $4$-adjacent faces in $J$. This reinterpretation of Rosenfeld's work, motivates the following definition

\begin{definition}
Let  $\mathcal{T}$ be a tiling and $\mathcal{D}_{\mathcal{T}}$ its digital version. 
\begin{enumerate}[\rm(1)]
    \item We say that two faces $C_1,C_2\in\mathcal F_{\mathcal T}$ are edge-adjacent (or that $C_1$ and $C_2$ are edge-neighbors) if $\mathcal E_{C_1}\cap\mathcal E_{C_2}\neq\emptyset$.
    \item We say that two faces $C_1,C_2\in\mathcal F_{\mathcal T}$ are vertex-adjacent  (or that $C_1$ and $C_2$ are vertex-neighbors) if $\mathcal V_{C_1}\cap\mathcal V_{C_2}\neq\emptyset$.
    \item A finite secuence of faces $(C_0,\dots, C_n)$ is an edge-path of faces (vertex-path of faces) if $C_i$ is edge-adjacent (vertex-adjacent) to $C_{i+1}$ and $C_{i-1}$, for every $i=1,\dots, n-1.$
    \item A set $Y\subset\mathcal D_{\mathcal T}\cap\mathcal F_{\mathcal T}$ is edge-connected (vertex-connected) if there is an edge-path (vertex-path) between any two elements $C_1,C_2\in Y$, such that every element of the path belongs to $Y$.
    \item A digital Jordan curve $J\subset \mathcal D_{\mathcal T}$ is an edge-Jordan curve, if $J$ is open and for every face $C\in J\cap\mathcal F_{\mathcal T}$ there are  exactly two faces in $J$ that are edge-adjacent with $C$.
\end{enumerate}
\end{definition}

Let  $\mathcal{T}$ be a tiling and define
$$\Delta(\mathcal T)=\sup\{|\mathcal F_{x}|\}_{ x\in\mathcal V_{\mathcal T}}.$$
Observe that if $J$ is an edge-Jordan curve and $\mathcal F_{x}\subset J$ for a certain $x\in \mathcal V_x$, then every element of $\mathcal F_x$ has exactly two edge-neighbors in $J$, and therefore $J$ cannot contain any other face of the tiling.  This yields the following remark.

\begin{remark}\label{r: grado acotado implica wb}
Let  $\mathcal T$ be a tiling and $J\subset\mathcal D_{\mathcal T}$ be an edge-Jordan curve. If $|J\cap\mathcal F_{\mathcal T}|\geq \Delta(\mathcal{T})+1$, then $J$ is well-behaved.
\end{remark}

We can now generalize Rosenfeld's theorem as follows.

\begin{theorem}\label{teo:cuadrosrosenfeld}
Let $\mathcal T$ be a tiling with $\Delta(\mathcal T)<\infty$. If $J\subset\mathcal{D}_{\mathcal{T}}$ is an edge-Jordan curve with $|J\cap\mathcal F_{\mathcal T}|\geq \Delta(\mathcal{T})+1,$ then $I(J)\cap \mathcal{F}_{\mathcal{T}}\neq\emptyset$, $O(J)\cap \mathcal{F}_{\mathcal{T}}\neq\emptyset$ and these two sets are vertex-connected.
\end{theorem}

\begin{proof} Since $I(J)\cup J$ is a finite set and $\mathcal F_{\mathcal T}$ is infinite, we always have that $O(J)\cap \mathcal{F}_{\mathcal{T}}\neq\emptyset$. On the other hand, 
by Remark~\ref{r: grado acotado implica wb} and Lemma~\ref{l:wb caras en interior}, we infer that  $I(J)\cap \mathcal{F}_{\mathcal{T}}\neq\emptyset$. 

Let us prove that $I(J)\cap\mathcal{F}_{\mathcal{T}}$ is vertex connected.
We know that $I(J)$ is connected and closed. Thus, for every $C\in \mathcal{F}_{\mathcal{T}}\cap I(J)$ we have that $\mathcal{V}_{C}\cup\mathcal{E}_{C}\subset I(J)$. 

\textit{Claim:} For every $A\in\mathcal{E}_{\mathcal{T}}\cap I(J)$,  $\mathcal{F}_{A}\cap I(J)\neq\emptyset$. Indeed,  if  $A\in\mathcal{E}_{\mathcal{T}}\cap I(J)$ is such that $\mathcal{F}_{A}\cap I(J)=\emptyset$, we can use the fact that  $O(J)$ is closed to conclude that  $\mathcal{F}_{A}\subset J$. Since $J$ is an open digital Jordan curve in $\mathcal D_{\mathcal T}$,  Proposition~\ref{prop:ultimaultimaprop} guarantees that $J\cap \mathcal V_{\mathcal T}=\emptyset$. Thus, each face of $A$ has a common edge with three faces of $J$, which contradicts the definition of $J$. This proves the claim.

 Since $I(J)$ is a connected Alexandrov space, Theorem \ref{teo:resumenconexidad} ensures that $I(J)$ is digitally arc connected. Thus for any $C_{1},C_{2}\in I(J)\cap \mathcal{F}_{\mathcal{T}}$ there is a digital arc of minimum length $\alpha_{0}=(x_0=C_{1}, x_1, \dots, x_n=C_{2})$ in $I(J)$ from $C_{1}$ to $C_{2}$. We can also assume that $\alpha_0$ contains a maximum number of faces. Since $\alpha_0$ has minimum length, there are at most three consecutive elements of the same tile in $\alpha_0$. Let
$$m:=\max\{k\in\mathbb{N}\mid \forall i\leq k:  x_{2i}\in\mathcal{F}_{\mathcal{T}}\}.$$ 

Suppose that  $m\neq\frac{n}{2}$. First observe that if $2m= n-1$, then $x_{2m}$ and $x_n$ would be two consecutive faces in $\alpha_0$, which is impossible. Thus,  $0\leq  2m<n-1$.  In this situation, $x_{2m}$ is a face of the tiling while $x_{2m+2}$ is not. If $x_{2m+1}\in\mathcal{E}_{x_{2m}}$, then $x_{2m+2}\in\mathcal{V}_{x_{2m}}$, contradicting the fact that $\alpha_0$ is a digital arc. Therefore we infer that  $x_{2m+1}\in\mathcal{V}_{x_{2m}}$, $x_{2m+2}\in\mathcal{E}_{\mathcal{T}}\setminus\mathcal{E}_{x_{2m}}$ and $x_{2m+1}\in\mathcal{V}_{x_{2m+2}}$.

By the claim, we can pick a face  $C\in\mathcal{F}_{x_{2m+2}}\cap I(J)$. Once again, since $\alpha_0$ is a digital arc, we conclude that $C\notin\alpha_0$ and $x_{2m+3}\in\mathcal{V}_{C}$ (if $2m+2<n$). Thus, $\{x_{2m_{0}+1}, x_{2m_{0}+2}, x_{2m_{0}+3}\}\subset\overline{\{C\}}$, implying that $x_{2m_{0}+4}\notin\overline{\{C_{x_{2m_{0}+2}}\}}$.

For $i\in\{0,\dots ,n\}$, let

\[   
y_i = 
     \begin{cases}
       x_i &\quad\text{if}\hspace{.2cm}i\neq 2m_{0}+2,\\
        C&\quad\text{if}\hspace{.2cm}i= 2m_{0}+2.\\
     \end{cases}
\]

and define $\alpha_1:=(y_0, y_1, \dots, y_n)$. Then $\alpha_1$ is a digital arc from $C_{1}$ to $C_{2}$ in $I(J)$ containing more faces than $\alpha_0$. This contradicts the fact that $\alpha_0$ contains a maximum number of faces. Thus, we can conclude that  $m=\frac{n}{2}$ and therefore $(x_0,x_2,\dots,x_{2m})$ is a simple vertex-path between $C_1$ and $C_2$. This proves that $I(J)\cap\mathcal{F}_{\mathcal{T}}$ is vertex-connected.

Similarly we can prove that $O(J)\cap\mathcal{F}_{\mathcal{T}}$ is vertex-connected. This completes the proof.
\end{proof}

It is clear that Theorem~\ref{teo:cuadrosrosenfeld} directly implies Rosenfeld's theorem  (Theorem~\ref{teo:Rosenfeld}).
Moreover, we can also deduce similar results for every grid of points induced by the faces of a regular tiling of the plane, such as  the one given by Kopperman in  \cite[Theorem 26]{Kopperman} for  a hexagonal grid of points.

Finally, there is a dual theorem, also due to Rosenfeld (see  \cite[Theorem 3.3]{Rosenfeld6}), that is obtained by interchanging the roles of $4$ and $8$ in Theorem \ref{teo:Rosenfeld}. To finish this paper, we present a generalization of this result in the context that we have been working on.
In order to do this, consider a tiling $\mathcal T$ and define a \textit{vertex-Jordan curve} as a digital Jordan curve $J\subset \mathcal D_{\mathcal T}$ such that $J\subset \mathcal F_{\mathcal T}\cup\mathcal V_{\mathcal T}$ and with the property that  every $C\in\mathcal{F}_{\mathcal{T}}\cap J$ has exactly  two vertex-neighbors in  $J$.

\begin{theorem}
Let $\mathcal{T}$ be a tiling  with $\Delta (\mathcal T)<\infty$. Suppose that  $J\subset\mathcal{D}_{\mathcal{T}}$ is a vertex-Jordan curve such that  $|J\cap \mathcal F_{\mathcal T}|\geq \Delta (\mathcal T)+1$. Then

\begin{enumerate}[\rm(1)]
    \item $J$ is well-behaved.
    \item $I(J)\cap \mathcal{F}_{\mathcal{T}}\neq\emptyset$ and $O(J)\cap \mathcal{F}_{\mathcal{T}}\neq\emptyset$.
    \item  $I(J)\cap \mathcal{F}_{\mathcal{T}}$ and $O(J)\cap \mathcal{F}_{\mathcal{T}}$ are edge-connected.
\end{enumerate}
\end{theorem}

\begin{proof}
Since $I(J)$ and $O(J)$ are the components of $Y:=I(J)\cup O(J)$, we can find two  sets $U, K\subset \mathcal D_{\mathcal T}$, with $U$ open and $K$ closed, such that
$$U\cap Y=I(J)=K\cap Y.$$
(1) If $J$ is not well-behaved, there exists a vertex $x\in\mathcal V_{\mathcal T}$, such that $\mathcal F_{x}\subset J$. Now, for every $C_1,C_2\in \mathcal F_{x}$, $C_1$ and $C_2$ are vertex-adjacent. Since $J$ is a vertex-Jordan curve and $|\mathcal F_{x}|\geq 3$, we infer that  $J\cap \mathcal F_{\mathcal T}=\mathcal F_x$. Thus
$$|\mathcal F_x|=|J\cap \mathcal F_{\mathcal T}|\geq \Delta (\mathcal T)+1\geq |\mathcal F_x|+1,$$
a contradiction.

(2) Since we always have that $O(J)\cap \mathcal{F}_{\mathcal{T}}\neq\emptyset$, we only need to prove the other inequality.  By contradiction, assume that $I(J)\cap \mathcal{F}_{\mathcal{T}}=\emptyset$. We claim that $\mathcal V_{\mathcal T}\cap I(J)\neq\emptyset$. Indeed, if $\mathcal V_{\mathcal T}\cap I(J)=\emptyset$, then every element of $I(J)$ is an edge. Thus, for every $E\in I(J)$, we get the following inclusions
$$\mathcal F_{E}\subset N(E)\subset U\subset I(J)\cup J\;\text{ and }\;\mathcal V_{E}\subset \overline{\{E\}}\subset K\subset I(J)\cup J.$$ This implies that $\mathscr A (E)=\mathcal V_{E}\cup\mathcal F_{E}\subset J$. Since $J$ is a digital Jordan curve, $J$ cannot contain any other element of $\mathcal D_{\mathcal T}$. This implies that $J$ contains only two faces, which contradicts the hypothesis $|J\cap \mathcal F_{\mathcal T}|\geq \Delta (\mathcal T)+1$. Therefore $\mathcal V_{\mathcal T}\cap I(J)\neq\emptyset$ and since $J$ is well behaved, we infer from Lema~\ref{l:wb caras en interior} that $I(J)$ contains a face of the tiling. This contradicts the original assumption that $I(J)\cap \mathcal{F}_{\mathcal{T}}=\emptyset$.

(3) Since $I(J)$ is connected, for any $C_1,C_2\in I(J)\cap \mathcal F_{\mathcal T}$, there exists a digital path $\alpha_0=(C_1=x_0, x_1,\dots, x_n=C_2)$ completely contained in $I(J)$. Furthermore,  we can assume that $\alpha_0$ contains a minimum number of vertices (namely, every other path between $C_1$ and $C_2$ contains at least the same amount of vertices than $\alpha_0$).

If $\alpha_0$ contains a vertex, let 
$$m:=\min\{k\in \{1,\dots, n\}\mid x_k\text{ is a vertex}\}.$$
In this case $x_{m-1}$ and $x_{m+1}$ belong to $\mathscr{A}(x_m)=\mathcal E_{x_m}\cup\mathcal F_{x_m}$. 

Observe that $\mathscr{A}(x_m)$ can only contain one element of $J$, at most. Indeed, if $D_1,D_2\in J\cap \mathscr{A}(x_m)$, then $\{D_1, D_2\}\subset \mathcal F_{x_m}$, because $J$ does not have any edge. This implies that $D_1$ and $D_2$ are vertex-adjacent, but since $x_{m}\notin J$, there must exist two other faces, $D, D'\in J$, different from $D_1$ and $D_2$, such that $D_1$ is vertex adjacent to $D$ and $D'$. Hence $J$ cannot be a vertex-Jordan curve, a contradiction.

Since $\mathscr{A}(x_m)$ induces a cycle in the connectedness graph of $\mathcal D_{\mathcal T}$, we can find a digital path $\beta=(x_{m-1}=y_0,y_1,\dots, y_k=x_{m+1})$, where each $y_i$ belongs to $\mathscr{A}(x_m)\setminus J$. In particular, there are no vertices in $\beta$. Furthermore, since $N(x_m)\subset U\subset I(J)\cup J$ and $\overline{\{x_m\}}\subset K\subset I(J)\cup J$, we get that 
$$\{y_1,\dots, y_k\}\subset \mathscr{A}(x_m)\setminus J\subset \big(I(J)\cup J\big)\setminus J=I(J).$$
Therefore the path 
$$\alpha_1:=(C_1=x_0,\dots ,x_{m-1},y_1,\dots, y_{k-1},x_{m+1},\dots ,x_n=C_2)$$
is a digital path between $C_1$ and $C_2$, which is completely contained in $I(J)$ and it contains one vertex less than $\alpha_0$, a contradiction. Thus $\alpha_0$ does not contain any vertex, and therefore  $x_{k}\in \mathcal F_{\mathcal T}$ if $k$ is even and  $x_{k}\in \mathcal E_{\mathcal T}$ if $k$ is odd. Then 
$$\alpha:=(C_1=x_0,x_2,\dots, x_n=C_2)$$ is an edge-path of faces, which proves that $I(J)$ is edge-connected, as desired. 
Analogously we can prove that $O(J)$ is edge-connected.
\end{proof}

\bibliographystyle{amsplain}

\end{document}